\documentclass[11pt]{article}
\usepackage{enumerate}
\usepackage{amssymb,a4wide,latexsym,makeidx,epsfig}
\usepackage{amsthm}
\usepackage{dsfont} %\mathds font
\usepackage{amsmath}
\usepackage{lipsum} % for filler text
\usepackage{enumerate}
\usepackage{mathrsfs}
\usepackage{xcolor}
\usepackage{tikz}%%%%draw graph%%%PDFLaTeX or PDFTeXify
\usepackage{geometry}
\usepackage{appendix}
\usepackage{multirow} %合并行
\usepackage[colorlinks, linkcolor=black, anchorcolor=black, citecolor=black]{hyperref}%参考文献等加链接。colorlinks表示采用颜色作为标识，若不用则采用默认的方框方式。citecolor选项表示颜色。
\usepackage{microtype} %allow LATEX to adjust spacing between characters to reduce the number of bad line breaks.
\usepackage{colonequals} % := use \colonequals
\newtheorem{theorem}{Theorem}[section]

\newtheorem{lemma}[theorem]{Lemma}

\newtheorem{corollary}[theorem]{Corollary}

\newtheorem{conjecture}[theorem]{Conjecture}
\newtheorem{claim}{Claim}[section]

\allowdisplaybreaks
\begin{document}
\textwidth 150mm \textheight 225mm
\title{On a conjecture of Nikiforov involving a spectral radius condition for a graph to contain all trees
\thanks{Supported by the National Natural Science Foundation of China (No. 11871398)
and China Scholarship Council (No. 202006290071).}
}
\author{{Xiangxiang Liu$^{1,2}$, Hajo Broersma$^{2,}$\thanks{Corresponding author.}, Ligong Wang$^{1}$}\\
{\small $^{1}$ School of Mathematics and Statistics,}\\ {\small Northwestern Polytechnical University, Xi'an, Shaanxi 710129, PR China}\\
{\small $^{2}$ Faculty of Electrical Engineering, Mathematics and Computer Science,}\\ {\small University of Twente, P.O. Box 217, 7500 AE Enschede, The Netherlands}\\
{\small E-mail: xxliumath@163.com; h.j.broersma@utwente.nl; lgwangmath@163.com}}
\date{}
\maketitle
\begin{center}
\begin{minipage}{120mm}
\vskip 0.3cm
\begin{center}
{\small {\bf Abstract}}
\end{center}
{\small
We partly confirm a Brualdi-Solheid-Tur\'{a}n type conjecture due to Nikiforov, which is a spectral radius analogue of the well-known Erd\H{o}s-S\'os Conjecture that any tree of order $t$ is contained in a graph of average degree greater than $t-2$. We confirm Nikiforov's Conjecture for all brooms and for a larger class of spiders. For our proofs we also obtain a new Tur\'{a}n type result which might turn out to be of independent interest.
}
{\small
\vskip 0.1in \noindent {\bf Keywords}: Brualdi-Solheid-Tur\'{a}n type problem, spectral radius, spider, broom \vskip
0.1in \noindent {\bf AMS Subject Classification (2020)}: \ 05C50, 05C35
}
\end{minipage}
\end{center}

\section{Introduction}
\label{sec:spiders-intro}
A central problem in extremal graph theory is the following Tur\'{a}n-type problem: for a given graph $H$, what is the maximum number of edges in an $H$-free graph with a given order? In the past decades, much attention has been paid to a spectral version of this question, that is, what is the maximum spectral radius of an $H$-free graph with a given order? The latter type of problem is called a Brualdi-Solheid-Tur\'{a}n type problem in \cite{Nikiforov} by Nikiforov. Examples of such problems are numerous since every Tur\'{a}n type problem gives rise to a corresponding Brualdi-Solheid-Tur\'{a}n type problem. As argued in \cite{Nikiforov}, ``the study of Brualdi-Solheid-Tur\'an type problems is an important topic in spectral graph theory''. Several groups of researchers have studied the relationship between the spectral radius and forbidden subgraphs (such as cliques, paths, cycles and complete bipartite subgraphs). We refer to \cite{Babai,GaoHou,Nikiforov1,Nikiforov2,Nikiforov,Wilf,Zhailin,ZhaiWang} for more information.

Motivated by these problems and earlier works, we study a conjecture due to Nikiforov \cite{Nikiforov}, which is a spectral radius analogue of the well-known Erd\H{o}s-S\'{o}s Conjecture that a graph of average degree greater than $t-2$ admits any tree of order $t$. Before we give more details concerning our work, we start by giving some essential definitions and introducing some useful notation.

Let $G=(V(G),E(G))$ be a simple undirected graph with vertex set $V(G)$ and edge set $E(G)$. We use $|G|\colonequals|V(G)|$ and $e(G)\colonequals|E(G)|$ to denote the order and size of $G$, respectively. Let $\mu(G)$ be the largest eigenvalue of the adjacency matrix $A(G)$ of $G$. We call $\mu(G)$ the \emph{spectral radius} of $G$. A graph $G$ is said to be \emph{$H$-free} if $H$ is not a subgraph of $G$. In order to avoid confusion, please note that we mean subgraph here and not induced subgraph. The \emph{Tur\'{a}n number} of $H$ is the maximum number of edges in an $H$-free graph of order $n$, and denoted by ex$(n,H)$. Given two disjoint graphs $G$ and $H$, the \emph{disjoint union} of $G$ and $H$, denoted by $G\cup H$, is the graph with vertex set $V(G)\cup V(H)$ and edge set $E(G)\cup E(H)$. We use $mG$ to denote the disjoint union of $m$ copies of $G$. The \emph{join} of $G$ and $H$, denoted by $G\vee H$, is the graph obtained from $G\cup H$ by adding edges joining every vertex of $G$ to every vertex of $H$.

Adopting the notation of \cite{Nikiforov}, let $S_{n,k}$ denote the graph obtained by joining every vertex of a complete graph $K_{k}$ to every vertex of an independent set of order $n-k$, that is, $S_{n,k}=K_{k}\vee\overline{K_{n-k}}$. Let $S_{n,k}^{+}$ be the graph obtained from $S_{n,k}$ by adding a single edge joining two vertices of the independent set of $S_{n,k}$. In addition, from \cite{Nikiforov} we know
\begin{equation}\label{eq:spectral radius1}
\mu(S_{n,k})=\frac{k-1}{2}+\sqrt{kn-\frac{3k^{2}+2k-1}{4}},
\end{equation}
and
\begin{equation}\label{eq:spectral radius2}
\mu(S_{n,k})<\mu(S_{n,k}^{+})<\mu(S_{n,k})+\frac{1}{n-k-2\sqrt{(n-k)/k}}.
\end{equation}

Based on the Erd\H{o}s-S\'{o}s Conjecture, in 2010 Nikiforov proposed the following Brualdi-Solheid-Tur\'{a}n type conjecture concerning trees.

\begin{conjecture}\label{conj:trees1}{\normalfont (\cite{Nikiforov})} Let $k\geq 2$ and let $G$ be a graph of sufficiently large order $n$.
\begin{itemize}
\item[{\rm (a)}] If $\mu(G)\geq \mu(S_{n,k})$, then $G$ contains all trees of order $2k+2$, unless $G=S_{n,k}$.
\item[{\rm (b)}] If $\mu(G)\geq \mu(S_{n,k}^{+})$, then $G$ contains all trees of order $2k+3$, unless $G=S_{n,k}^{+}$.
\end{itemize}
\end{conjecture}

In \cite{Nikiforov}, Nikiforov proved that Conjecture~\ref{conj:trees1} holds for paths.

\begin{theorem}\label{th:path}{\normalfont (\cite{Nikiforov})}
Let $k\geq 2$ and let $G$ be a graph of sufficiently large order $n$.
\begin{itemize}
\item[{\rm (a)}] If $\mu(G)\geq \mu(S_{n,k})$, then $G$ contains a $P_{2k+2}$, unless $G=S_{n,k}$.
\item[{\rm (b)}] If $\mu(G)\geq \mu(S_{n,k}^{+})$, then $G$ contains a $P_{2k+3}$, unless $G=S_{n,k}^{+}$.
\end{itemize}
\end{theorem}

Recently, Hou et al. \cite{HouLiu} proved that Conjecture~\ref{conj:trees1} (a) holds for all trees of diameter at most four. Liu, Broersma and Wang \cite{LiuBW} proved that Conjecture~\ref{conj:trees1} (b) holds for all trees of diameter at most four, except for the subdivision of $K_{1, k+1}$ in which every edge is subdivided precisely once.

In this paper, we study Conjecture~\ref{conj:trees1} for spiders. A \emph{spider} is a tree with at most one vertex of degree more than 2. The vertex of degree more than 2 is called the \emph{center} of the spider (if all vertices have degree 1 or 2, so if the spider is a path, then any designated vertex can act as its center). A \emph{leg} of a spider is a path from the center to a vertex of degree 1. The \emph{length} of a leg is the number of edges of the leg. We use $S(t_1, t_2, \ldots, t_m)$ to denote a spider consisting of one designated center and $m$ legs with lengths $t_1, t_2, \ldots, t_m$; see Figure~\ref{fig:spider} for an example. Thus $S(t_1, t_2, \ldots, t_m)$ has $1+\sum_{i=1}^{m} t_i$ vertices and $\sum_{i=1}^{m} t_i$ edges.

\begin{figure}[htbp]
\begin{center}
\begin{tikzpicture}[scale=0.06,auto,swap]
\tikzstyle{vertex}=[circle,draw=black,fill=black]
\node[vertex,scale=0.5] (v0) at (0,0) {}; \draw (0,6) node {$v_0$};
\node[vertex,scale=0.5] (v11) at (-30,-15) {}; \draw (-38,-15) node {$v_{1,1}$};
\node[vertex,scale=0.5] (v21) at (-10,-15) {}; \draw (-18,-15) node {$v_{2,1}$};
\node[vertex,scale=0.5] (v31) at (10,-15) {}; \draw (2,-15) node {$v_{3,1}$};
\node[vertex,scale=0.5] (v41) at (30,-15) {}; \draw (22,-15) node {$v_{4,1}$};
\node[vertex,scale=0.5] (v12) at (-30,-30) {}; \draw (-38,-30) node {$v_{1,2}$};
\node[vertex,scale=0.5] (v22) at (-10,-30) {}; \draw (-18,-30) node {$v_{2,2}$};
\node[vertex,scale=0.5] (v32) at (10,-30) {}; \draw (2,-30) node {$v_{3,2}$};
\node[vertex,scale=0.5] (v13) at (-30,-45) {}; \draw (-38,-45) node {$v_{1,3}$};
\node[vertex,scale=0.5] (v23) at (-10,-45) {}; \draw (-18,-45) node {$v_{2,3}$};

\draw[thick] (v0)--(v11); \draw[thick] (v0)--(v21); \draw[thick] (v0)--(v31); \draw[thick] (v0)--(v41);
\draw[thick] (v11)--(v12); \draw[thick] (v21)--(v22); \draw[thick] (v31)--(v32); \draw[thick] (v12)--(v13); \draw[thick] (v22)--(v23);
\end{tikzpicture}
\caption{The spider $S(3,3,2,1)$.}\label{fig:spider}
\end{center}
\vspace{-0.5cm}
\end{figure}
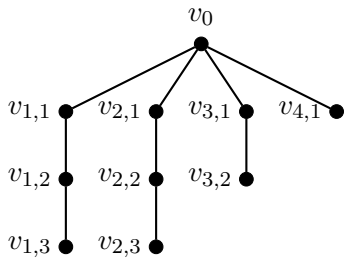

The aforementioned Erd\H{o}s-S\'{o}s Conjecture has been confirmed for several classes of spiders in a series of papers; see \cite{Fan,FanHuo,FanSun}. Our first contribution is the following spectral radius result on the existence of a class of spiders. By imposing two restrictions on the number of odd legs, i.e., of odd length, we can prove the following theorem involving spiders.

\begin{theorem}\label{th:spider3} Let $k\geq 2$ and let $S$ be a spider of order $2k+3$ with $r$ odd legs and $s$ legs of length 1. If $r\geq 3$, $2s-r\geq 2$ and $n$ is sufficiently large, then every graph $G$ of order $n$ with $\mu(G)\geq \mu(S_{n,k})$ contains $S$ as a subgraph.
\end{theorem}

We postpone all proofs to later sections. Theorem~\ref{th:spider3} confirms Conjecture~\ref{conj:trees1} (b) for all spiders satisfying the condition in the statement of the theorem. Since $\mu(S_{n,k})<\mu(S^+_{n,k})$, Theorem~\ref{th:spider3} is in fact a stronger result. Let $S$ be a spider of order $2k+2$ with $r$ odd legs and $s$ legs of length 1 such that $r\geq 2$ and $2s-r\geq 1$. Let $S'$ be the graph obtained from $S$ by adding an extra pendant edge at the center of $S$. Then $S'$ is a spider of order $2k+3$ with $r+1$ odd legs and $s+1$ legs of length 1 such that $r+1\geq 3$ and $2(s+1)-(r+1)\geq 2$. Applying Theorem~\ref{th:spider3}, we immediately derive the following result which confirms Conjecture~\ref{conj:trees1} (a) for a class of spiders.

\begin{corollary}\label{co:spider2} Let $k\geq 2$ and let $S$ be a spider of order $2k+2$ with $r$ odd legs and $s$ legs of length 1. If $r\geq 2$, $2s-r\geq 1$ and $n$ is sufficiently large, then every graph $G$ of order $n$ with $\mu(G)\geq \mu(S_{n,k})$ contains $S$ as a subgraph.
\end{corollary}

Our next result deals with the special subclass of spiders called brooms. For $s,t\geq1$, a \emph{broom} $B_{s,t}$ is a tree on $s+t$ vertices obtained by identifying the center of a star $K_{1,s}$ and an end-vertex of a path $P_{t}$; see Figure~\ref{fig:broom}. Note that the broom $B_{s,t}$ can be viewed as a spider $S(t_1, t_2, \ldots, t_{s+1})$, where $t_1=\cdots=t_{s}=1$ and $t_{s+1}=t-1$. Moreover, $B_{1,t}=P_{t+1}$ and $B_{s,1}=K_{1,s}$. It is easy to check that if $n$ is large enough, then $S_{n,k}$ contains all brooms of order $2k+3$ except for $B_{1,2k+2}$ and $B_{2,2k+1}$, and $S_{n,k}^{+}$ contains all brooms of order $2k+3$ except for $B_{1,2k+2}$.

\begin{figure}[htbp]
\begin{center}
\begin{tikzpicture}[scale=0.06,auto,swap]
\tikzstyle{vertex}=[circle,draw=black,fill=black]
\node[vertex,scale=0.5] (vt) at (66,-10) {}; \draw (66,-16) node {$t$};
\node[vertex,scale=0.5] (vt-1) at (46,-10) {}; \draw (46,-16) node {$t-1$};
\node[vertex,scale=0.5] (vt-2) at (26,-10) {};
\node[vertex,scale=0.5] (v3) at (-6,-10) {};
\node[vertex,scale=0.5] (v2) at (-26,-10) {}; \draw (-26,-16) node {$2$};
\node[vertex,scale=0.5] (v1) at (-46,-10) {}; \draw (-46,-16) node {$1$};
\node[vertex,scale=0.5] (u1) at (46,10) {}; \draw (46,16) node {1};
\node[vertex,scale=0.5] (u2) at (62,10) {}; \draw (62,16) node {2};
\node[vertex,scale=0.5] (us) at (86,10) {}; \draw (86,16) node {$s$};

\draw[dotted,thick] (-3,-10)--(23,-10); \draw[dotted,thick] (65,10)--(83,10);

\draw[thick] (v1)--(v2); \draw[thick] (v2)--(v3); \draw[thick] (vt-2)--(vt-1); \draw[thick] (vt-1)--(vt);
\draw[thick] (u1)--(vt); \draw[thick] (u2)--(vt); \draw[thick] (us)--(vt);
\end{tikzpicture}
\caption{The broom $B_{s,t}$.}\label{fig:broom}
\end{center}
\vspace{-0.5cm}
\end{figure}
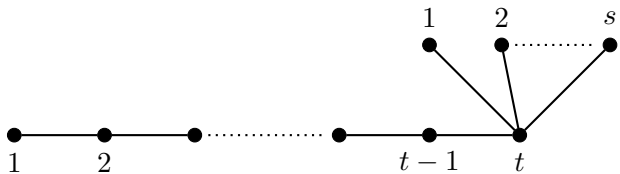

The next theorem confirms Conjecture~\ref{conj:trees1} for all brooms.

\begin{theorem}\label{th:brooms} Let $k\geq 2$ and let $G$ be a graph of sufficiently large order $n$.
\begin{itemize}
\item[{\rm (a)}] If $\mu(G)\geq \mu(S_{n,k})$, then $G$ contains all brooms of order $2k+2$, unless $G=S_{n,k}$.
\item[{\rm (b)}] If $\mu(G)\geq \mu(S_{n,k}^{+})$, then $G$ contains all brooms of order $2k+3$, unless $G=S_{n,k}^{+}$.
\end{itemize}
\end{theorem}

In order to prove Theorem~\ref{th:brooms}, we shall use Theorems~\ref{th:path}, \ref{th:spider3} and the following Tur\'{a}n type result for connected graphs involving the broom $B_{2,2k+1}$. Theorem~\ref{th:B22K1} below might be of independent interest.

\begin{theorem}\label{th:B22K1} For $k\geq 2$ and $n$ sufficiently large, let $G$ be a connected graph of order $n$. If $e(G)\geq e(S_{n,k}^{+})=kn-\frac{k(k+1)}{2}+1$, then $G$ contains $B_{2,2k+1}$ as a subgraph.
\end{theorem}

The remainder of this paper is organized as follows. In the next section, we provide some auxiliary results that will be used in our proofs. In Section~\ref{sec:proof spider3}, we prove Theorem~\ref{th:spider3}. Section~\ref{sec:proof broomB22K1} is devoted to our proof of Theorem~\ref{th:brooms}. In Section~\ref{sec:proof B22K1}, we prove Theorem~\ref{th:B22K1}. Finally, we conclude this paper with some remarks on generalized brooms and by presenting some open problems in Section~\ref{sec:concluding}.

%%%%%%%%%%%%%%%%%%%%%%%%%%%%%%%%%%%%%%%%%%%%%%%%%%%%%%%%%%%%%%%%%%%%%%%%%%%%%%%%%%%
%%%%%%%%%%%%%%%%%%%%%%%%%%%%%%%%%%%%%%%%%%%%%%%%%%%%%%%%%%%%%%%%%%%%%%%%%%%%%%%%%%%

\section{Preliminaries}\label{sec:spiders-pre}

In this section, we provide some additional terminology and lemmas that we will use. Let $G$ be a connected graph. For any vertex $u\in V(G)$, let $N^{d}(u)\colonequals\{v\in V(G)\colon\, d_{G}(v,u)=d\}$, where $d_{G}(v,u)$ is the distance between $u$ and $v$ in $G$. Let $d_{G}(u)$ be the degree of $u$ in $G$ and let $\delta(G)$ be the minimum degree of $G$. For a non-empty subset $U\subseteq V(G)$, let $G[U]$ be the subgraph of $G$ induced by $U$, $E(U)$ be the edge set of $G[U]$, and $e(U)\colonequals|E(U)|$. For two disjoint vertex sets $U, V\subseteq V(G)$, let $E(U,V)$ be the set of edges in $G$ with one end-vertex in $U$ and one end-vertex in $V$, and let $e(U,V)\colonequals|E(U,V)|$. Given a path $P=v_1v_2\cdots v_t$, we denote the sub-path $v_iv_{i+1}\cdots v_j$ by $v_iPv_j$. All logarithms in this paper are to the base 2. We use the standard Bachmann-Landau notation to indicate asymptotic growth rates of functions.

We will use the following known lemma on matrices in the set-up of the proof of Theorem~\ref{th:spider3}. Given an $n\times n$ matrix $A$, let $A_{ij}$ be the $(i,j)$-th entry of $A$ for $1\leq i,j\leq n$.

\begin{lemma}\label{le:trees1}{\normalfont (\cite{GaoHou})} Given $a,b\in \mathbb{Z}^{+}$ and an $n\times n$ nonnegative symmetric irreducible matrix $A$, let $\mu$ be the largest eigenvalue of $A$ and $\mu'$ be the largest root of $f(x)=x^{2}-ax-b$. Define $B=f(A)=A^{2}-aA-bI$ and let $B_{j}=\sum_{i=1}^{n}B_{ij}$ {\rm(}$1\leq j\leq n${\rm)}. If $B_{j}\leq0$ for all $j\in \{1,2,\ldots,n\}$, then $\mu\leq \mu'$, with equality holding if and only if $B_{j}=0$ for all $j\in \{1,2,\ldots,n\}$.
\end{lemma}

A \emph{linear forest} is a forest all whose components are paths. We shall use the following known results on the Tur\'{a}n numbers of paths and linear forests in the proof of Theorem~\ref{th:spider3}.

\begin{lemma}\label{le:Pt}{\normalfont (\cite{ErGa59})}
For any positive integers $t$ and $n$, we have ${\rm ex}(n, P_t)\leq\frac{t-2}{2}n.$
\end{lemma}

\begin{lemma}\label{le:kP3}{\normalfont (see \cite[Theorem~2.2]{BuKe})}
For any integer $\ell \geq 2$ and sufficiently large $n$, we have ${\rm ex}(n, \ell P_3)<\left(\ell-\frac{1}{2}\right)n.$
\end{lemma}

\begin{lemma}\label{le:linearforest1}{\normalfont (see \cite[Theorem~2]{LidiLiu})}
For any integer $\ell \geq 2$, let $F=\bigcup_{1\leq i\leq \ell}P_{a_i}$ be a linear forest with $a_i\geq 2$ for all $i\in [\ell]$. If at least one $a_i$ is not 3, then for $n$ sufficiently large,
$${\rm ex}(n, F)\leq \Bigg(\sum_{1\leq i\leq \ell}\left\lfloor\frac{a_i}{2}\right\rfloor-1\Bigg)n.$$
\end{lemma}

Next we introduce two lemmas that will be used in the proof of Theorem~\ref{th:B22K1}.

\begin{lemma}\label{le:sumpath}{\normalfont (\cite{BBRS})} Let $G$ be a graph and for each vertex $v\in V(G)$, let $p_v$ be the length of a longest path in $G$ starting at $v$. Then
$e(G)\leq \sum_{v\in V(G)}\frac{p_v}{2}$.
\end{lemma}

\begin{lemma}\label{le:P2K3}{\normalfont (\cite{BGLS})} For $k\geq 2$ and $n$ sufficiently large, let $G$ be a connected graph of order $n$ with $G\neq S^{+}_{n,k}$. If $e(G)\geq e(S_{n,k}^{+})=kn-\frac{k(k+1)}{2}+1$, then $G$ contains $P_{2k+3}$ as a subgraph.
\end{lemma}

We shall also apply the following partial solution of the Erd\H{o}s-S\'{o}s Conjecture.

\begin{lemma}\label{le:spiders}{\normalfont (\cite{FanSun})}
If $G$ is a graph on $n$ vertices with $e(G)>\frac{(t-2)n}{2}$, then $G$ contains every $t$-vertex spider with three legs.
\end{lemma}

We end this section with some known results about $\mu(G)$.

\begin{lemma}\label{le:sepctral-degree}{\normalfont (\cite{Nikiforov3})} If $G$ is a graph with $n$ vertices, $m$ edges and $\delta(G)=\delta$, then
$$\mu(G)\leq \frac{\delta-1}{2}+\sqrt{2m-\delta n +\frac{(\delta+1)^{2}}{4}}.$$
\end{lemma}

\begin{lemma}\label{le:sepctral-degree2}{\normalfont (\cite{Stan})} If $G$ is a graph with $m$ edges, then
$$\mu(G)\leq -\frac{1}{2}+\sqrt{2m+\frac{1}{4}}.$$
\end{lemma}

\begin{lemma}\label{le:sepctral-subgraph}{\normalfont (\cite{Nikiforov})} Let the numbers $c\geq0$, $k\geq2$ and $n\geq2^{4k}$, and let $G$ be a graph of order $n$. If $\delta(G)\leq k-1$ and $\mu(G)\geq \frac{k-1}{2}+\sqrt{kn-k^{2}+c}$, then there exists a subgraph $H$ of $G$ satisfying one of the following conditions:
\begin{itemize}
\item[{\rm (i)}] $\mu(H)>\sqrt{(2k+1)|H|}$;
\item[{\rm (ii)}] $|H|\geq \sqrt{n}$, $\delta(H)\geq k$ and $\mu(H)>\frac{k-1}{2}+\sqrt{k|H|-k^{2}+c+\frac{1}{2}}$.
\end{itemize}
\end{lemma}

%%%%%%%%%%%%%%%%%%%%%%%%%%%%%%%%%%%%%%%%%%%%%%%%%%%%%%%%%%%%%%%%%%%%%%%%%%%%%%%%%%%
%%%%%%%%%%%%%%%%%%%%%%%%%%%%%%%%%%%%%%%%%%%%%%%%%%%%%%%%%%%%%%%%%%%%%%%%%%%%%%%%%%%

\section{Proof of Theorem~\ref{th:spider3}}\label{sec:proof spider3}

Let $G$ be a graph with $V(G)=[n]$ and $\mu(G)\geq \mu(S_{n,k})$, where $n$ is sufficiently large. For a contradiction, suppose that $G$ contains no copy of $S$. Without loss of generality, we may assume that $S=S(t_1, t_2, \ldots, t_m)$, where $t_i=1$ for all $1\leq i\leq s$, $t_i\geq 3$ is odd for all $s+1\leq i\leq r$, and $t_{i}\geq 2$ is even for all $r+1\leq i\leq m$.

Let $f(x)=x^{2}-(k-1)x-k(n-k)$. Note that $\mu(S_{n,k})$ is the largest root of $f(x)$. Let $B=f(A(G))=A(G)^{2}-(k-1)A(G)-k(n-k)I$ and $B_{v}=\sum_{i=1}^{n}B_{iv}$ for any $v\in V(G)$. By Lemma~\ref{le:trees1} and since $\mu(G)\geq \mu(S_{n,k})$, there exists a vertex $u\in V(G)$ with $B_{u}\geq 0$. Let $L_{u}$ be the graph with vertex set $N^{1}(u)\cup N^{2}(u)$ and edge set $E(N^{1}(u))\cup E(N^{1}(u),N^{2}(u))$. Since the $(i,u)$-entry of $A(G)^2$ is the number of walks of length 2 between $i$ and $u$,
we have
\begin{equation}\label{eq:trees1}
B_{u}=\sum\limits_{x\in N^{1}(u)}d_{L_{u}}(x)-(k-2)d_{G}(u)-k(n-k).
\end{equation}

We complete the proof by first proving the following claim and then distinguishing two cases based on the degree of $u$.

\begin{claim}\label{cl:trees1} $d_{G}(u)\geq k+1$.
\end{claim}

\begin{proof}
Since $d_{L_{u}}(x)\leq n-2$ for all $x\in N^{1}(u)$, using (\ref{eq:trees1}) we have
\begin{equation}\label{eq:trees2}
0\leq B_{u}\leq d_{G}(u)(n-2)-(k-2)d_{G}(u)-k(n-k)=(d_{G}(u)-k)(n-k).
\end{equation}
Thus $d_{G}(u)\geq k$. If $d_{G}(u)\geq k+1$, then we are done. If $d_{G}(u)=k$, then inequality (\ref{eq:trees2}) implies $0=B_{u}=d_{G}(u)(n-2)-(k-2)d_{G}(u)-k(n-k)$. This implies that $d_{L_{u}}(x)=n-2$ for all $x\in N^{1}(u)$. Then $G$ contains $S_{n,k}$ as a subgraph.

Since $S$ is a bipartite graph, we may assume that $(V', V'')$ is a bipartition of $V(S)$, say with $|V'|\geq |V''|$. Then $|V'|-|V''|=r-1\geq 2$. Moreover, since $|V(S)|=2k+3$ is odd, we have that $|V'|-|V''|$ is odd, so $r-1\geq 3$. Then $|V''|\leq k$. Thus $S$ is a subgraph of $S_{n,k}$. This contradicts the assumption that $G$ contains no copy of $S$.
\end{proof}

We divide the rest of the proof into two cases: (i) $k+1\leq d_{G}(u)\leq \log n$ and (ii) $d_{G}(u)>\log n$.

\medskip\noindent
{\bf Case 1.}  $k+1\leq d_{G}(u)\leq \log n$.
\vspace{0.05cm}

\noindent
In this case, since $B_u\geq 0$ and by equality (\ref{eq:trees1}), we have
\begin{align*}
e(N^{1}(u),N^{2}(u))&=\sum\limits_{x\in N^{1}(u)}d_{L_{u}}(x)-2e(N^{1}(u))\\
&>(k-2)d_G(u)+k(n-k)-d_{G}^{2}(u)=kn-o(n).
\end{align*}
Thus $|N^{2}(u)|>\frac{e(N^{1}(u),N^{2}(u))}{|N^{1}(u)|}>(1-o(1))\frac{kn}{\log n}$.

We claim that there are $\frac{n}{2\log n}$ vertices in $N^{2}(u)$ each of which has at least $k$ neighbors in $N^{1}(u)$. Otherwise, we have $e(N^{1}(u),N^{2}(u))<\frac{n}{2\log n}|N^{1}(u)|+(|N^{2}(u)|-\frac{n}{2\log n})(k-1)<\frac{n}{2}+(k-1)n\ll kn-o(n)$, a contradiction.

Therefore, there exist $\frac{n}{2\log n}\frac{1}{{|N^{1}(u)|\choose k}}>\frac{n}{\log^{k+1} n}$ vertices in $N^{2}(u)$ which have $k$ common neighbors in $N^{1}(u)$. Without loss of generality, let $X\subseteq N^{1}(u)$ and $Y\subseteq N^{2}(u)$ with $|X|=k$ and $|Y|=\frac{n}{\log^{k+1} n}\gg 2k+2$ such that $X$ is completely joined to $Y$. Similarly as in the last paragraph of the proof of Claim~\ref{cl:trees1}, we derive that $G[X\cup Y]$ contains a copy of $S$, a contradiction.

\medskip\noindent
{\bf Case 2.} $d_{G}(u)>\log n$.
\vspace{0.05cm}

\noindent
We consider two subcases based on the number of edges between $N^{1}(u)$ and $N^{2}(u)$.

\medskip\noindent
{\bf Subcase 2.1.} $e(N^{1}(u),N^{2}(u))>kd_{G}(u)+(2k-2)|N^{2}(u)|-k(n-k)$.
\vspace{0.05cm}

\noindent
By equality (\ref{eq:trees1}) and since $B_{u}\geq0$, we have $\sum_{x\in N^{1}(u)}d_{L_{u}}(x)\geq(k-2)d_{G}(u)+k(n-k)$. Then
\begin{align}\label{eq:2.1}
e(L_{u})&=~\frac{1}{2}\Bigg(\sum_{x\in N^{1}(u)}d_{L_{u}}(x)+e(N^{1}(u),N^{2}(u))\Bigg) \nonumber\\
&>~\frac{1}{2}((k-2)d_{G}(u)+k(n-k)+kd_{G}(u)+(2k-2)|N^{2}(u)|-k(n-k)) \nonumber\\
&=~(k-1)(d_{G}(u)+|N^{2}(u)|)=(k-1)|V(L_{u})|.
\end{align}

We next claim that $L_u$ contains the following linear forest.

\begin{claim}\label{cl:linearforest1} There is a $\bigcup_{s+1\leq i\leq m}P_{t_i}$ in $L_u$ such that each $P_{t_i}$ has an end-vertex in $N^{1}(u)$.
\end{claim}

\begin{proof}
If $t_i=2$ for all $s+1\leq i\leq m$, then $s=r\geq 3$. Since $2k+3$ is odd, we have that $r$ is even. Hence, we further have $s=r\geq 4$. Thus $m-s=(2k+3-1-s)/2\leq k-1$. By Lemmas~\ref{le:Pt} and \ref{le:linearforest1}, for sufficiently large $N$, we have ${\rm ex}\left(N, (m-s)P_{2}\right)\leq(k-2)N$. By inequality~(\ref{eq:2.1}), there is
an $(m-s)P_{2}$ in $L_u$. By the definition of $L_u$, each $P_{2}$ has an end-vertex in $N^{1}(u)$.

Next, we assume that at least one $t_i$ ($s+1\leq i\leq m$) is not 2. We now show that there is a $\bigcup_{s+1\leq i\leq m}P_{t_i+1}$ in $L_u$. If $r=s$ and $m-s=m-r=1$, then $t_m+1\leq 2k+3-1-r+1\leq 2k$ since $r\geq 3$. By Lemma~\ref{le:Pt} and inequality~(\ref{eq:2.1}), there is a $P_{t_m+1}$ in $L_u$. If $m=r$ and $m-s=r-s=1$, then $s\geq 3$ since $2\leq 2s-r=2s-(s+1)$. In this case, we also have $t_m+1\leq 2k+3-1-s+1\leq 2k$. Thus there is a $P_{t_m+1}$ in $L_u$ by Lemma~\ref{le:Pt} and inequality~(\ref{eq:2.1}). If $m-s\geq 2$, then by Lemma~\ref{le:linearforest1} and since $2s-r\geq 2$, for sufficiently large $N$, we have
\begin{align*}
 &{\rm ex}\left(N, \bigcup_{s+1\leq i\leq m}P_{t_i+1}\right)\leq \left(\sum_{s+1\leq i\leq m}\left\lfloor\frac{t_i+1}{2}\right\rfloor-1\right)N \\
 = &\left(\sum_{s+1\leq i\leq r}\frac{t_i+1}{2}+\sum_{r+1\leq i\leq m}\frac{t_i}{2}-1\right)N
 =\left(\sum_{s+1\leq i\leq m}\frac{t_i}{2}+\sum_{s+1\leq i\leq r}\frac{1}{2}-1\right)N \\
 =&\left(\frac{\sum_{s+1\leq i\leq m}t_i+\sum_{s+1\leq i\leq r}1}{2}-1\right)N
 =\left(\frac{2k+3-1-s+r-s}{2}-1\right)N \\
 =&~\frac{2k+r-2s}{2}N \leq ~\frac{2k-2}{2}N=~(k-1)N.
\end{align*}
By inequality~(\ref{eq:2.1}), there is a $\bigcup_{s+1\leq i\leq m}P_{t_i+1}$ in $L_u$. For each $s+1\leq i\leq m$, let $v'_i$ and $v''_i$ be the first and second vertex along the path $P_{t_i+1}$, respectively. By the definition of $L_u$, at least one of $v'_i$ and $v''_i$ is contained in $N^{1}(u)$. Hence, there is a $\bigcup_{s+1\leq i\leq m}P_{t_i}$ in $L_u$ such that each $P_{t_i}$ has an end-vertex in $N^{1}(u)$.
\end{proof}

By Claim~\ref{cl:linearforest1}, there is a $\bigcup_{s+1\leq i\leq m}P_{t_i}$ in $L_u$ such that each $P_{t_i}$ has an end-vertex in $N^{1}(u)$. Together with vertex $u$ and $s$ additional vertices in $N^{1}(u)$, this $\bigcup_{s+1\leq i\leq m}P_{t_i}$ forms a copy of $S$ in $G$, a contradiction.

\medskip\noindent
{\bf Subcase 2.2.} $e(N^{1}(u),N^{2}(u))\leq kd_{G}(u)+(2k-2)|N^{2}(u)|-k(n-k)$.
\vspace{0.05cm}

\noindent
By equality (\ref{eq:trees1}) and since $B_{u}\geq0$, we have $\sum_{x\in N^{1}(u)}d_{L_{u}}(x)\geq(k-2)d_{G}(u)+k(n-k)$. Since $d_{G}(u)>\log n$ and $n$ is sufficiently large, we have
\begin{align}\label{eq:2.2}
e(N^{1}(u))&=~\frac{1}{2}\Bigg(\sum_{x\in N^{1}(u)}d_{L_{u}}(x)-e(N^{1}(u),N^{2}(u))\Bigg) \nonumber\\
&\geq~\frac{1}{2}((k-2)d_{G}(u)+k(n-k)-kd_{G}(u)-(2k-2)|N^{2}(u)|+k(n-k)) \nonumber\\
&=~-d_{G}(u)+k(n-|N^{2}(u)|)-k^{2}+|N^{2}(u)| \nonumber\\
&\geq~-d_{G}(u)+k(d_{G}(u)+1)-k^{2}+|N^{2}(u)| \nonumber\\
&=~(k-1)d_{G}(u)+k-k^{2}+|N^{2}(u)|>\frac{2k-3}{2}|N^{1}(u)|.
\end{align}

We next claim that $G[N^{1}(u)]$ contains the following linear forest.

\begin{claim}\label{cl:linearforest2} There is a $\bigcup_{s+1\leq i\leq m}P_{t_i}$ in $G[N^{1}(u)]$.
\end{claim}

\begin{proof}
If $t_i=3$ for all $s+1\leq i\leq m$, then $m=r$. Since $2s-r\geq 2$, we have $s\geq r-s+2=m-s+2$. We first show that $m-s\leq k/2$. Indeed, if $m-s>k/2$, then $s\geq m-s+2>(k+4)/2$. On the other hand, $s=2k+3-1-3(m-s)< 2k+2-3k/2=(k+4)/2$, a contradiction. Hence, $m-s\leq k/2$. By Lemmas~\ref{le:Pt} and \ref{le:kP3}, we have ${\rm ex}(N, (m-s)P_{3})\leq (k/2-1/2)N\leq (2k-3)N/2$ for sufficiently large $N$. By inequality~(\ref{eq:2.2}), there is an $(m-s)P_{3}$ in $G[N^{1}(u)]$.

Next, we assume that at least one $t_i$ ($s+1\leq i\leq m$) is not 3. If $r=s$ and $m-s=m-r=1$, then $t_m\leq 2k+3-1-r\leq 2k-1$ since $r\geq 3$. By Lemma~\ref{le:Pt} and inequality~(\ref{eq:2.2}), there is a $P_{t_m}$ in $G[N^{1}(u)]$. If $m=r$ and $m-s=r-s=1$, then $s\geq 3$ since $2\leq 2s-r=2s-(s+1)$. In this case, we also have $t_m\leq 2k+3-1-s\leq 2k-1$. Thus there is a $P_{t_m}$ in $G[N^{1}(u)]$ by Lemma~\ref{le:Pt} and inequality~(\ref{eq:2.2}). If $m-s\geq 2$, then by Lemma~\ref{le:linearforest1} and since $r\geq 3$, for sufficiently large $N$, we have
\begin{align*}
 &{\rm ex}\left(N, \bigcup_{s+1\leq i\leq m}P_{t_i}\right)\leq\left(\sum_{s+1\leq i\leq m}\left\lfloor\frac{t_i}{2}\right\rfloor-1\right)N \\
 =&\left(\sum_{s+1\leq i\leq r}\frac{t_i-1}{2}+\sum_{r+1\leq i\leq m}\frac{t_i}{2}-1\right)N
 =\left(\sum_{s+1\leq i\leq m}\frac{t_i}{2}-\sum_{s+1\leq i\leq r}\frac{1}{2}-1\right)N \\
 =&\left(\frac{\sum_{s+1\leq i\leq m}t_i-\sum_{s+1\leq i\leq r}1}{2}-1\right)N
 =\left(\frac{2k+3-1-s-(r-s)}{2}-1\right)N \\
 =&~\frac{2k-r}{2}N \leq~\frac{2k-3}{2}N.
\end{align*}
By inequality~(\ref{eq:2.2}), there is a $\bigcup_{s+1\leq i\leq m}P_{t_i}$ in $G[N^{1}(u)]$.
\end{proof}

By Claim~\ref{cl:linearforest2}, there is a $\bigcup_{s+1\leq i\leq m}P_{t_i}$ in $G[N^{1}(u)]$. Together with vertex $u$ and $s$ additional vertices in $N^{1}(u)$, this $\bigcup_{s+1\leq i\leq m}P_{t_i}$ forms a copy of $S$ in $G$, a contradiction. This contradiction completes our proof of Theorem~\ref{th:spider3}.

%%%%%%%%%%%%%%%%%%%%%%%%%%%%%%%%%%%%%%%%%%%%%%%%%%%%%%%%%%%%%%%%%%%%%%%%%%%%%%%%%%%
%%%%%%%%%%%%%%%%%%%%%%%%%%%%%%%%%%%%%%%%%%%%%%%%%%%%%%%%%%%%%%%%%%%%%%%%%%%%%%%%%%%
\section{Proof of Theorem~\ref{th:brooms}}\label{sec:proof broomB22K1}

In this section, we give our proof of Theorem~\ref{th:brooms}, which confirms Conjecture~\ref{conj:trees1} for all brooms. We first prove that the result holds for $B_{2,2k+1}$.

\begin{theorem}\label{th:broomB22K1} For integers $k\geq2$ and $n$ sufficiently large, every graph $G$ of order $n$ with $\mu(G)\geq \mu(S_{n,k}^{+})$ contains $B_{2,2k+1}$ as a subgraph.
\end{theorem}

\begin{proof}
We first consider the case that $G$ is connected. If $e(G)\geq kn-k(k+1)/2+1$, then $G$ contains $B_{2,2k+1}$ as a subgraph by Theorem~\ref{th:B22K1}. Next we assume that $e(G)\leq kn-k(k+1)/2$. If $\delta(G)\geq k$, then by Lemma~\ref{le:sepctral-degree} and equality~(\ref{eq:spectral radius1}), we have
\begin{align*}
\mu(G)&\leq~\frac{\delta(G)-1}{2}+\sqrt{2e(G)-\delta(G)n +\frac{(\delta(G)+1)^{2}}{4}}\\
&\leq~\frac{k-1}{2}+\sqrt{2e(G)-kn +\frac{(k+1)^{2}}{4}}\\
&\leq~\frac{k-1}{2}+\sqrt{2\left(kn-\frac{k(k+1)}{2}\right)-kn +\frac{(k+1)^{2}}{4}}\\
&=~\mu(S_{n,k})<\mu(S_{n,k}^{+}),
\end{align*}
which is a contradiction. Thus $\delta(G)\leq k-1$.

Let $c=(k-1)^{2}/4$. By Lemma~\ref{le:sepctral-subgraph}, there is a subgraph $H$ of $G$ satisfying either (i) $\mu(H)>\sqrt{(2k+1)|H|}$ or (ii) $|H|\geq \sqrt{n}$, $\delta(H)\geq k$ and $\mu(H)>(k-1)/2+\sqrt{k|H|-k^{2}+c+1/2}$.

If (i) holds, then by Lemma~\ref{le:sepctral-degree2}, we have $\sqrt{(2k+1)|H|}<\mu(H)\leq \sqrt{2e(H)}$. So $(2k+1)|H|<2e(H)$. Thus $G$ contains $B_{2,2k+1}$ as a subgraph by Lemma~\ref{le:spiders}.

If (ii) holds, then $H$ contains a component $H'$ with $\delta(H')\geq k$ and $\mu(H')=\mu(H)>(k-1)/2+\sqrt{k|H|-k^{2}+c+1/2}\geq (k-1)/2+\sqrt{k|H'|-k^{2}+c+1/2}$. Moreover, since $|H|\geq \sqrt{n}$ and $\mu(H)>(k-1)/2+\sqrt{k|H|-k^{2}+c+1/2}$, the order of $H'$ is large when $n$ is large.
Next we will prove $\mu(H')>(k-1)/2+\sqrt{k|H'|-(3k^{2}+2k-3)/4}>\mu\big(S_{|H'|,k}^{+}\big)$. For a contradiction, suppose that $(k-1)/2+\sqrt{k|H'|-(3k^{2}+2k-3)/4}\leq \mu\big(S_{|H'|,k}^{+}\big)$. By inequality~(\ref{eq:spectral radius2}), we have
\begin{align*}
\mu\big(S_{|H'|,k}^{+}\big)&<~\mu\big(S_{|H'|,k}\big)+\frac{1}{|H'|-k-2\sqrt{(|H'|-k)/k}}\\
&=~\frac{k-1}{2}+\sqrt{k|H'|-\frac{3k^{2}+2k-1}{4}}+\frac{1}{|H'|-k-2\sqrt{(|H'|-k)/k}}.
\end{align*}
Hence,
$$\sqrt{k|H'|-\frac{3k^{2}+2k-3}{4}}-\sqrt{k|H'|-\frac{3k^{2}+2k-1}{4}}<\frac{1}{|H'|-k-2\sqrt{(|H'|-k)/k}},$$
so
$$\frac{1/2}{\sqrt{k|H'|-(3k^{2}+2k-3)/4}+\sqrt{k|H'|-(3k^{2}+2k-1)/4}}<\frac{1}{|H'|-k-2\sqrt{(|H'|-k)/k}}.$$
Since $|H'|$ is sufficiently large, the above inequality is impossible. Hence, $\mu(H')>\mu\big(S_{|H'|,k}^{+}\big)$. By an analogous argument as in the first paragraph of the proof, we can derive a contradiction. This contradiction completes the proof for connected $G$.

If $G$ is disconnected, then $G$ contains a component $G'$ with $\mu(G')=\mu(G)$. Then $\mu(G')\geq \mu(S_{n,k}^{+})\geq \mu\big(S_{|G'|,k}^{+}\big)$. Thus we can apply the above arguments to $G'$ and complete the proof.
\end{proof}

Now we have all ingredients to present our proof of Theorem~\ref{th:brooms}.

\begin{proof}[Proof of Theorem~\ref{th:brooms}] For positive integers $s$ and $t$, let $B_{s,t}$ be a broom. Note that $B_{1,t}$ is in fact a path. In this case, the result follows from Theorem~\ref{th:path}. Hence, we may assume that $s\geq 2$ in the following arguments.

The broom $B_{s,t}$ can be viewed as a spider $S(t_1, t_2, \ldots, t_{s+1})$, where $t_1=\cdots=t_{s}=1$ and $t_{s+1}=t-1$. Such a spider has at least $s$ legs of length 1. Let $r$ be the number of odd legs. So $s\leq r\leq s+1$. We first prove Theorem~\ref{th:brooms} (a). Since $s\geq 2$, we have $2s-r\geq 2s-(s+1)\geq 1$. By Corollary~\ref{co:spider2}, the result holds. We next prove Theorem~\ref{th:brooms} (b). The case for $B_{2,2k+1}$ follows from Theorem~\ref{th:broomB22K1}. For the case $s\geq 3$, we have $2s-r\geq 2s-(s+1)\geq 2$. By Theorem~\ref{th:spider3}, the result holds.
\end{proof}

%%%%%%%%%%%%%%%%%%%%%%%%%%%%%%%%%%%%%%%%%%%%%%%%%%%%%%%%%%%%%%%%%%%%%%%%%%%%%%%%%%%
%%%%%%%%%%%%%%%%%%%%%%%%%%%%%%%%%%%%%%%%%%%%%%%%%%%%%%%%%%%%%%%%%%%%%%%%%%%%%%%%%%%

\section{Proof of Theorem~\ref{th:B22K1}}
\label{sec:proof B22K1}

Let $G$ be a connected graph of order $n$ with $e(G)\geq e(S_{n,k}^{+})=kn-k(k+1)/2+1$. We first consider the case $G=S^+_{n,k}$. Let $v_1, v_2, \ldots, v_k$ denote the $k$
vertices with degree $n-1$, $u_1, u_2, \ldots, u_{n-k-2}$ the $n-k-2$ vertices with degree $k$, and $w_1, w_2$ the two vertices with degree $k+1$. Then $w_1w_2v_1u_1v_2u_2\cdots v_ku_k$ is a path of order $2k+2$. Together with the edge $v_ku_{k+1}$, this path forms a $B_{2,2k+1}$ in $G$.

Next, we consider the case $G\neq S^+_{n,k}$. For a contradiction, suppose that $G$ contains no $B_{2,2k+1}$. Let $P=v_1v_2\cdots v_{\ell}$ be a longest path in $G$. By Lemma~\ref{le:P2K3}, we have $\ell\geq 2k+3$. Let $X=V(P)$ and $Y=V(G)\setminus V(P)$. We next state and prove two claims.

\begin{claim}\label{cl:B22K1-largel} $2k+3\leq \ell \leq 4k-1$.
\end{claim}

\begin{proof} We first suppose $4k\leq \ell \leq n-1$. Since $P$ is a longest path, we have $uv_1\notin E(G)$ and $uv_{\ell}\notin E(G)$ for any vertex $u\in Y$. Since $G$ is connected, there is an $i\in \{2, 3, \ldots, \ell-1\}$ and a vertex $w\in Y$ such that $v_iw\in E(G)$. Since $\ell\geq 4k$, at least one of the paths $v_1\cdots v_i$ and $v_i\cdots v_{\ell}$ has order at least $2k+1$, say $v_1\cdots v_i$. Together with the edge $v_iw$, the path $v_1\cdots v_iv_{i+1}$ forms a graph that contains $B_{2,2k+1}$ as a subgraph, a contradiction.

We next suppose $\ell=n$. Since $e(G)\geq kn-k(k+1)/2+1>n$, there exist two vertices $v_i$ and $v_j$ with $i+2\leq j$ and $\{i, j\}\neq \{1, n\}$ such that $v_iv_j\in E(G)$. Since $n$ is sufficiently large, at least one of the paths $v_1\cdots v_i$, $v_i\cdots v_j$ and $v_j\cdots v_n$ has order greater than $2k+1$. It is easy to check that there is a $B_{2,2k+1}$ in $G$, a contradiction.
\end{proof}

\begin{claim}\label{cl:B22K1-observation} Let $v$ be any vertex in $Y$. Then
\begin{itemize}
\item[{\rm (i)}] there is no edge between $v$ and $\{v_1, v_2, \ldots, v_{\ell-2k}\}\cup\{v_{2k+1}, v_{2k+2}, \ldots, v_{\ell}\}$; and
\item[{\rm (ii)}] $v$ has at most $\left\lceil\frac{4k-\ell}{2}\right\rceil$ neighbors in $X$.
\end{itemize}
\end{claim}

\begin{proof} Since $P$ is a longest path, there is no edge between $v$ and $\{v_1, v_{\ell}\}$. Since $G$ contains no $B_{2, 2k+1}$, there is no edge between $v$ and $\{v_2, \ldots, v_{\ell-2k}\}\cup\{v_{2k+1}, \ldots, v_{\ell-1}\}$. Thus (i) holds. For any $1\leq i\leq \ell-1$, at most one of $vv_i$ and $vv_{i+1}$ is an edge of $G$; otherwise the path $v_1\cdots v_ivv_{i+1}\cdots v_{\ell}$ is longer than $P$. Combining with (i), the vertex $v$ has at most $\lceil(\ell-2(\ell-2k))/2\rceil=\lceil(4k-\ell)/2\rceil$ neighbors in $X$. Thus (ii) holds.
\end{proof}

If  $e(X, V(H))+e(H)\leq (k-1/2)|H|$ for each component $H$ in $G-X$, then $e(G)\leq (k-1/2)(n-\ell)+{\ell \choose 2}<kn-k(k+1)/2+1$ when $n$ is sufficiently large, a contradiction. Thus there exists a component $H$ in $G-X$ with $e(X, V(H))+e(H)> (k-1/2)|H|$. In the remainder of this proof, let $H$ denote such a component. Our aim is to show that $e(X,V(H))+e(H)\le (k-1/2)|H|$, contradicting the above. For any vertex $v\in V(H)$, let $s_v$ be the number of neighbors of $v$ in $X$, and let $p_v$ be the length of a longest path in $H$ starting at $v$ (so this longest path has $p_v+1$ vertices). We next state and prove three claims.

\begin{claim}\label{cl:B22K1-twovertices} For any vertex $v\in V(H)$, we have $vv_{\ell-2k+1}\notin E(G)$ and $vv_{2k}\notin E(G)$.
\end{claim}

\begin{proof} By symmetry, we only prove $vv_{\ell-2k+1}\notin E(G)$ for any $v\in V(H)$. For a contradiction, suppose that there is a vertex $u\in V(H)$ with $uv_{\ell-2k+1}\in E(G)$. Note that $|H|\geq 2$; otherwise if $|H|=1$, then $e(X, V(H))+e(H)=s_u\leq k-1\leq (k-1/2)|H|$ by Claims~\ref{cl:B22K1-largel} and \ref{cl:B22K1-observation} (ii), a contradiction. Before we derive at a contradiction, we first make some other observations.

We first observe that $p_v\leq \ell-2k-1$ for any vertex $v\in V(H)$. For any vertex $v\in V(H)$ with $vv_{\ell-2k+1}\in E(G)$, we have $p_v\leq \ell-2k-1$ in order to avoid a $P_{\ell+1}$. For any vertex $v\in V(H)$ with $vv_{\ell-2k+1}\notin E(G)$, if $p_v\geq \ell-2k$, then we may assume that $vw_1w_2\cdots w_{\ell-2k}$ is a path in $H$. In order to avoid a $P_{\ell+1}$ or a $B_{2, 2k+1}$, we have $u\notin \{v, w_1, w_2, \ldots, w_{\ell-2k}\}$, and there is no path in $H$ connecting $u$ and the path $vw_1w_2\cdots w_{\ell-2k}$. This contradicts the fact that $H$ is connected.

We next observe that there is at most one vertex in $H$ which is adjacent to $\lceil(4k-\ell)/2\rceil$ vertices in $X$. Otherwise, suppose that there are two vertices $v', v''\in V(H)$ with $s_{v'}=s_{v''}=\lceil(4k-\ell)/2\rceil$. This implies that $v'$ (resp., $v''$) is either adjacent to both $v_{\ell-2k+1}$ and $v_{\ell-2k+3}$ or adjacent to both $v_{\ell-2k+2}$ and $v_{\ell-2k+4}$. Since $H$ is connected, there is a path $Q$ in $H$ connecting $v'$ and $v''$. It is easy to check that $G[X\cup V(Q)]$ contains a path longer than $P$, a contradiction.

Therefore, by Lemma~\ref{le:sumpath} and Claim~\ref{cl:B22K1-observation} (ii), we have
\begin{align*}
&~e(X, V(H))+e(H)\leq~\sum_{v\in V(H)}\left(s_v+\frac{p_v}{2}\right)\\
\leq&~\left(\left\lceil\frac{4k-\ell}{2}\right\rceil+\frac{\ell-2k-1}{2}\right)+(|H|-1)\left(\left\lceil\frac{4k-\ell}{2}\right\rceil-1+\frac{\ell-2k-1}{2}\right)\\
\leq&~\left(k-\frac{1}{2}\right)|H|,
\end{align*}
a contradiction.
\end{proof}

\begin{claim}\label{cl:B22K1-svn0} For any vertex $v\in V(H)$, if $s_v\neq 0$, then $s_v+\frac{p_v}{2}\leq \min\left\{k-\frac{1}{2}, \frac{\ell-1}{2}-\frac{p_v}{2}\right\}$.
\end{claim}

\begin{proof} By Claim~\ref{cl:B22K1-observation} (i) and Claim~\ref{cl:B22K1-twovertices}, we have $s_v\leq \lceil(4k-\ell-2)/2\rceil$. Note that $k-1/2\leq (\ell-1)/2-p_v/2$ when $p_v\leq \ell-2k$, and $(\ell-1)/2-p_v/2< k-1/2$ when $p_v\geq \ell-2k+1$. If $p_v\leq \ell-2k$, then $s_v+p_v/2\leq \lceil(4k-\ell-2)/2\rceil+(\ell-2k)/2\leq k-1/2$. Next, we show that $2(s_v+p_v)\leq \ell-1$, which implies $s_v+p_v/2=s_v+p_v-p_v/2\leq (\ell-1)/2-p_v/2$ and thus completes the proof. Let $v_{i_1}, \ldots, v_{i_{s_v}}$ be all the neighbors of $v$ in $X$, where $i_1<\cdots<i_{s_v}$. In order to avoid a $P_{\ell+1}$ in $G$, we have $i_{s_v}\leq \ell-(p_v+1)$ and $i_1\geq p_v+2$. Thus $i_{s_v}-i_1\leq \ell-2p_v-3$. On the other hand, we have $\lceil(i_{s_v}-i_1+1)/2\rceil\geq s_v$. Thus $i_{s_v}-i_1\geq 2s_v-2$. Hence, $2s_v-2\leq \ell-2p_v-3$. This implies that $2(s_v+p_v)\leq \ell-1$.
\end{proof}

\begin{claim}\label{cl:B22K1-2k+3} If $\ell=2k+3$, then for any vertex $v\in V(H)$, we have $vv_{k+2}\notin E(G)$.
\end{claim}

\begin{proof} The case $k=2$ follows from Claim~\ref{cl:B22K1-twovertices}. We next consider the case $k= 3$. In this case, for any vertex $v\in V(H)$, $v_5$ is the only possible neighbor of $v$ in $X$. Let $V'\colonequals\{v\in V(H)\colon\, vv_5\in E(G)\}$ and $V''\colonequals V(H)\setminus V'$. In order to avoid a $B_{2, 2k+1}$, each vertex $v\in V'$ (resp., $v\in V''$) has at most two neighbors in $V'$ (resp., $V(H)$). Thus $e(X, V(H))+e(H)\leq e(X, V')+e(V')+e(V',V'')+e(V'')\leq |V'|+(\sum_{v\in V'}2)/2+\sum_{v\in V''}2=2|H|< (k-1/2)|H|$, a contradiction. In the following arguments, we may assume that $k\geq 4$.

For a contradiction, suppose that there is a vertex $u\in V(H)$ with $uv_{k+2}\in E(G)$. Let $\mathcal{H}$ be the set of all components of $H-u$ and let $H'\in \mathcal{H}$. For any vertex $v\in V(H')$, let $p'_v$ be the length of a longest path in $H'$ starting at $v$. Then $p'_v\leq p_v$. In the following, we show that $s_v+p'_v/2+\mathds{1}_{(vu\in E(G))}\leq k-1/2$ for any vertex $v\in V(H')$, where $\mathds{1}_{(vu\in E(G))}\colonequals 1$ if $vu\in E(G)$, and $\mathds{1}_{(vu\in E(G))}\colonequals 0$ if $vu\notin E(G)$.

We first consider a vertex $v\in V(H')$ with $s_v\neq 0$. By Claim~\ref{cl:B22K1-svn0}, we have $s_v+p'_v/2\leq s_v+p_v/2\leq \min\left\{k-1/2, k+1-p_v/2\right\}$. If $vu\notin E(G)$ (resp., $p_v\geq 5$), then $s_v+p'_v/2+\mathds{1}_{(vu\in E(G))}\leq s_v+p_v/2 \leq k-1/2$ (resp., $s_v+p'_v/2+\mathds{1}_{(vu\in E(G))}\leq k+1-5/2+1 \leq k-1/2$). Now we consider the case $vu\in E(G)$ and $p_v\leq 4$. In order to avoid a $P_{2k+4}$, there is no edge between $v$ and $\{v_{k}, v_{k+1}, v_{k+3}, v_{k+4}\}$. Combining with Claim~\ref{cl:B22K1-observation} (i) and Claim~\ref{cl:B22K1-twovertices}, we have $s_v\leq 2\lceil(k-5)/2\rceil+1\leq k-3$. Moreover, if we further have $p_v=4$, then $vv_5, vv_{2k-1}\notin E(G)$ for avoiding a $P_{2k+4}$. Thus $s_v\leq \max\{2\lceil(k-6)/2\rceil+1, 1\}\leq \max\{k-4, 1\}$ in this case. Hence, if $k\geq 5$,
then $s_v+p'_v/2+1\leq s_v+p_v/2+1\leq \max\{k-3+3/2, k-4+4/2\}+1=k-1/2$. If $k=4$, then $p_u\leq 4$ for avoiding a $P_{2k+4}$, and thus $p'_v\leq 3$. Hence, if $k=4$, then $s_v+p'_v/2+1\leq 1+3/2+1=k-1/2$.

We next consider a vertex $v\in V(H')$ with $s_v=0$. We now show that $p'_v\leq 2k-2$. Otherwise, suppose that $Q$ is a path in $H'$ of length $2k-1$ starting at $v$. Let $w$ be the other end-vertex of $Q$. Since $H$ is connected and $H'$ is a component of $H-u$, there is a path $Q'$ connecting $u$ and $Q$. We choose such a path $Q'$ with the minimum length, that is, $|V(Q)\cap V(Q')|=1$, say $V(Q)\cap V(Q')=\{x\}$. Then one of the paths $vQx$ and $xQw$ has length at least $k$, say $vQx$. Then $v_1\cdots v_{k+2}uQ'xQv$ is a path of length at least $2k+3$. This contradiction implies $p'_v\leq 2k-2$. In order to avoid a $P_{2k+4}$, if $p'_v\geq k$, then $vu\notin E(G)$. Thus $s_v+p'_v/2+\mathds{1}_{(vu\in E(G))}\leq 0+\max\{(2k-2)/2, (k-1)/2+1\}\leq k-1/2$.

From the above arguments and by Lemma~\ref{le:sumpath}, we have
$e(X, V(H))+e(H)= e(\{u\},X)+\sum_{H'\in \mathcal{H}}\left(e(V(H'), X)+e(H')+e(V(H'), \{u\})\right)\leq s_u+\sum_{H'\in \mathcal{H}}\sum_{v\in V(H')}(s_v+p'_v/2+\mathds{1}_{(vu\in E(G))})$ $\leq k-2+\sum_{H'\in \mathcal{H}}(k-1/2)|V(H')|< (k-1/2)|H|$,
a contradiction.
\end{proof}

Let $A\colonequals \{v\in V(H)\colon\, s_v\neq 0\}$, $B\colonequals \{v\in V(H)\colon\, s_v=0 \text{ and } p_v\leq 2k-1\}$ and $C\colonequals \{v\in V(H)\colon\, s_v=0 \text{ and } p_v\geq 2k\}$. Then for any vertex $v\in A\cup B$, we have $s_v+p_v/2 \leq k-1/2$ by Claim~\ref{cl:B22K1-svn0}. The next claim deals with vertices in $C$.

\begin{claim}\label{cl:B22K1-degree} If $C\neq \emptyset$, then for any vertex $v\in C$, the degree of $v$ in $H$ is at most $k-1$.
\end{claim}

\begin{proof} By the definition of $C$, we may assume that $w_0w_1w_2\cdots w_{2k}$ is a path in $H$, where $w_0\colonequals v$. Let $W=\{w_1, w_2, \ldots, w_{2k}\}$. Since $G$ is connected and $s_v=0$, there is a path $Q$ connecting a vertex of $X$ and a vertex of $W$. We choose such a path with minimal length, that is, $Q$ has exactly one common vertex with $X$ and exactly one common vertex with $W$. We may assume that $Q=u_0u_1\cdots u_t$ for some $t\geq 1$, where $u_0\in X$ and $u_t=w_i$ for some $1\leq i\leq 2k$.

We first show that $\ell\geq 2k+4$. Otherwise if $\ell=2k+3$, then $vv_{k+2}\notin E(G)$ by Claim~\ref{cl:B22K1-2k+3}. Thus one of the paths $v_1Pu_0$ and $u_0Pv_{\ell}$ has order at least $k+3$. In order to avoid a $P_{2k+4}$, we have $k+3+t+i\leq 2k+3$ and $k+3+t+2k-i\leq 2k+3$. Thus $k+t\leq i\leq k-t$, which is impossible. Hence, $\ell\geq 2k+4$.

Note that one of the paths $v_1Pu_0$ and $u_0Pv_{\ell}$ has order at least $\lceil(\ell+1)/2\rceil$, say $v_1Pu_0$. In order to avoid a $P_{\ell+1}$, we have $\lceil(\ell+1)/2\rceil+t+i\leq \ell$ and $\lceil(\ell+1)/2\rceil+t+2k-i\leq \ell$. Thus $2k+t-\lfloor(\ell-1)/2\rfloor\leq i\leq \lfloor(\ell-1)/2\rfloor-t$. Moreover, for any $0\leq j\leq j_0\colonequals\lceil(\ell+1)/2\rceil+t+i-2k-1$, the path $v_1Pu_0Qw_iw_{i-1}\cdots w_{j}$ has order at least $2k+1$. Note that $j_0=\lceil(\ell+1)/2\rceil+t+i-2k-1\geq \lceil(\ell+1)/2\rceil+t+2k+t-\lfloor(\ell-1)/2\rfloor-2k-1\geq 2$. In order to avoid a $B_{2, 2k+1}$, the vertex $v$ has no neighbor in $\{w_2, \ldots, w_{j_0}\}$ and has at most one neighbor in $\{w_{i+1}, \ldots, w_{2k}\}\cup (V(H)\setminus (W\cup V(Q)))$. Since $\ell\geq 2k+4$, the number of neighbors of $v$ in $H$ is at most $|\{w_1\}|+|\{w_{j_0+1}, \ldots, w_i\}|+t-1+1=i-j_0+t+1=2k+2-\lceil(\ell+1)/2\rceil\leq k-1$.
\end{proof}

Recall that for any vertex $v\in A\cup B$, we have $s_v+p_v/2 \leq k-1/2$. By Lemma~\ref{le:sumpath} and Claim~\ref{cl:B22K1-degree}, we have
\begin{align*}
e(X, V(H))+e(H)=&~e(X, A)+e(A\cup B)+e(C)+e(C, A\cup B)\\
\leq&~\sum_{v\in A\cup B}\left(s_v+\frac{p_v}{2}\right)+\sum_{v\in C}d_H(v)\\
\leq&~|A\cup B|\left(k-\frac{1}{2}\right)+|C|(k-1)\\
\leq&~\left(k-\frac{1}{2}\right)|H|.
\end{align*}
This contradiction completes the proof of Theorem~\ref{th:B22K1}.

%%%%%%%%%%%%%%%%%%%%%%%%%%%%%%%%%%%%%%%%%%%%%%%%%%%%%%%%%%%%%%%%%%%%%%%%%%%%%%%%%%%
%%%%%%%%%%%%%%%%%%%%%%%%%%%%%%%%%%%%%%%%%%%%%%%%%%%%%%%%%%%%%%%%%%%%%%%%%%%%%%%%%%%

\section{Concluding remarks}\label{sec:concluding}

In this paper, we proved that a Brualdi-Solheid-Tur\'{a}n type conjecture due to Nikiforov (Conjecture~\ref{conj:trees1}) holds for a class of spiders. In particular, we
also confirmed Conjecture~\ref{conj:trees1} for all brooms. For integers $s\geq1$ and $t\geq \ell\geq 1$, a \emph{generalized broom} $B^{\ell}_{s,t}$ is a tree on $s+t$ vertices obtained from a path $P_t$ by attaching $s$ pendant edges at the $\ell$-th vertex along the path. Note that $B^{\ell}_{s,t}=B^{t+1-\ell}_{s,t}$ and $B^1_{s,t}=B^t_{s,t}=B_{s,t}$. Using Theorem~\ref{th:spider3}, it is easy to derive the following result for generalized brooms.

\begin{corollary}\label{co:genbroom} For integers $k\geq2$, let $\mathcal{T}$ {\rm (}resp., $\mathcal{T}'${\rm)} be the set of all generalized brooms $B^{\ell}_{s,t}$ of order $2k+3$ with $s\geq 3$ {\rm (}resp., of order $2k+2$ with $s\geq 2${\rm)}. Then every graph $G$ of sufficiently large order $n$ with $\mu(G)\geq \mu(S_{n,k})$ contains all graphs in $\mathcal{T}$ and $\mathcal{T}'$.
\end{corollary}

\begin{proof} For $s\geq 3$, the generalized broom $B^{\ell}_{s,t}$ can be viewed as a spider $S(t_1, t_2, \ldots, t_{s+2})$, where $t_1=\cdots=t_{s}=1$, $t_{s+1}=\ell-1$ and $t_{s+2}=t-\ell$. Such a spider has at least $s$ legs of length 1. Let $r$ be the number of odd legs. So $s\leq r\leq s+2$. Since $s\geq 3$, we have $r\geq 3$. If $s\geq 4$, then $2s-r\geq 2s-(s+2)\geq 2$. If $s=3$, then $r=s+1$ since $2k+3$ is odd, so we also have $2s-r\geq 2$. By Theorem~\ref{th:spider3}, $G$ contains all graphs in $\mathcal{T}$. Since every graph $B^{\ell}_{s,2k+2-s}$ ($s\geq2$) is a subgraph of $B^{\ell}_{s+1,2k+3-(s+1)}$, we can further deduce that $G$ contains all graphs in $\mathcal{T}'$.
\end{proof}

An interesting and natural question is to study Conjecture~\ref{conj:trees1} for generalized brooms $B^{\ell}_{s,t}$ of order $2k+3$ with $s\leq 2$ (resp., of order $2k+2$ with $s=1$) and $2\leq \ell\leq t-1$. Another direction is to study Conjecture~\ref{conj:trees1} for other classes of spiders. Hopefully this will also lead to new ideas and approaches for resolving Conjecture~\ref{conj:trees1} for general trees.

\end{document}